\documentclass[10pt]{amsart}

\usepackage{amssymb,amsmath}
\usepackage{bbm}
\usepackage{amsthm}
\usepackage{enumerate}
\usepackage{comment}
\usepackage{mathtools}
\usepackage{xcolor}
\usepackage[title]{appendix}

\usepackage{todonotes}
\usepackage{amsaddr}

\mathtoolsset{showonlyrefs}
\newcommand{\R}{\mathbb{R}}
\newcommand{\C}{\mathbb{C}}
\newcommand{\Z}{\mathbb{Z}}
\newcommand{\N}{\mathbb{N}}

\newcommand{\torus}{\mathbb{T}}

\newcommand{\T}{\mathbb{T}}
\newcommand{\G}{\mathcal{G}}
\newcommand{\Sp}{\mathcal{S}}
\newcommand{\F}{\mathcal{F}}
\newcommand{\mi}{\mathrm{i}}

\newcommand{\dd}{\partial}
\newcommand{\dbar}{\bar{\partial}}

\newcommand{\graphg}{\mathfrak{G}}
\newcommand{\graphv}{\mathfrak{V}}
\newcommand{\graphe}{\mathfrak{E}}

\newcommand{\etal}{\textit{et al}.}

\DeclareMathOperator{\vol}{vol}
\DeclareMathOperator{\dg}{deg}
\DeclareMathOperator{\len}{len}

\newtheorem{thm}{Theorem}[section]
\newtheorem{lem}[thm]{Lemma}
\newtheorem{cor}[thm]{Corollary}
\newtheorem{prop}[thm]{Proposition}

\newtheorem{thmlet}{Theorem}

\theoremstyle{definition}
\newtheorem{defn}[thm]{Definition}
\newtheorem{ex}[thm]{Example}

\theoremstyle{remark}
\newtheorem{rem}[thm]{Remark}

\numberwithin{equation}{section}

\usepackage[style=numeric, backend=biber, sorting=nyt,doi=false, isbn=false, eprint=false,  url=false,giveninits=true]{biblatex}
\begin{document}

%%% In the title, use a double backslash "\\" to show a linebreak:
%%% Use one of the following two forms:
%%% \title{Text of the title}
%%% or
%%% \title[Short form for the running head]{Text of the title}
\title{$L^2$-stability analysis for Gabor phase retrieval}

%%% If there are multiple authors, they're described one at a time:
%%% First author: \author{} \address{} \curraddr{} \email{} \thanks{}
%%% Second author: \author{} \address{} \curraddr{} \email{} \thanks{}
%%% Third author: \author{} \address{} \curraddr{} \email{} \thanks{}

\author{Philipp Grohs}
\address{Faculty of Mathematics, University of Vienna, Oskar Morgenstern Platz 1, 1090 Vienna\\ Austria and Research Network DataScience@UniVie, University of Vienna, Kolingasse 14-16, 1090 Vienna, Austria\\
Johann Radon Institute for Computational and Applied Mathematics,
Austrian Academy of Sciences, Altenbergerstrasse 69, 4040 Linz, Austria.}
\email{philipp.grohs@univie.ac.at}
\author{Martin Rathmair}
\address{Faculty of Mathematics, University of Vienna, Oskar Morgenstern Platz 1, 1090 Vienna, Austria}
\email{martin.rathmair@univie.ac.at}

%%% Current address is optional.
% \curraddr{}

%%% Email address is optional.
% \email{}

%%% If there's a second author:
% \author{}
% \address{}
% \curraddr{}
% \email{}

%%% To have the current date inserted, use \date{\today}:
\date{\today}

%%% To include an abstract, uncomment the following two lines and type
%%% the abstract in between them:
\begin{abstract}
We consider the problem of reconstructing the missing phase information from spectrogram data $|\mathcal{G} f|,$
with 
$$
\mathcal{G}f(x,y)=\int_\mathbb{R} f(t) e^{-\pi(t-x)^2}e^{-2\pi i t y}dt,
$$ 
the Gabor transform of a signal $f\in L^2(\mathbb{R})$.
More specifically, we are interested in domains $\Omega\subseteq \mathbb{R}^2$, which allow for stable local reconstruction, that is 
$$
|\mathcal{G}g| \approx |\mathcal{G}f| \quad \text{in} ~\Omega
\quad\Longrightarrow \quad 
\exists \tau\in\mathbb{T}:\quad
\mathcal{G}g \approx \tau\mathcal{G}f \quad \text{in} ~\Omega.
$$
In recent work [P. Grohs and M. Rathmair. Stable Gabor Phase Retrieval and Spectral Clustering. Comm. Pure Appl. Math. (2019)] and 
[P. Grohs and M. Rathmair. Stable Gabor phase retrieval for multivariate functions. J. Eur. Math. Soc. (2021)]
we established a characterization of the stability of this phase retrieval problem in terms of the connectedness of the observed measurements. 
The main downside of the aforementioned results is that the similarity of two spectrograms is measured w.r.t. a first order weighted Sobolev norm. In this article we remove this flaw and essentially show that the Sobolev norm may be replaced by the $L^2-$norm.
Using this result allows us to show that it suffices to sample the spectrogram on suitable discrete sampling sets -- a property of crucial importance for practical applications.
\end{abstract}

\maketitle
%\tableofcontents

\section{Introduction}
Phase retrieval refers to the general problem of determining a signal from phase-less measurements. 
Oftentimes, the operator describing these measurements is linear and nicely invertible. However, due to the absence of phase information one has to deal with a nonlinear inverse problem, which is typically significantly harder to analyze.\\
Perhaps the most prominent example is the \emph{Fourier phase retrieval} problem, where one asks to determine a function $f\in L^2(\mathbb{R}^d)$ (up to a unimodular factor) from 
the modulus of its Fourier transform, that is from $\big|\hat{f}\big|$.
This question is obviously badly ill-posed, as one can attach any arbitrary phase information to $|\hat{f}|$. 
Even, if the same question is posed on the space $L^2[-1,1]$ of compactly supported functions there still remains a vast amount of amibiguity \cite{MR80802,1053648,doi:10.1080/713817747}.
Possible strategies to overcome these issues of non-uniqueness include incorporating a-priori information on the signals under consideration (e.g., $f$ is compactly supported and non-negative) and/or 
changing the measurement setup in order to introduce redundancy (e.g., replace the Fourier transform by the short-time Fourier transform).\\
Phase retrieval problems are encountered in a number of highly important applications in various fields such as astronomy \cite{dainty}, diffraction imaging \cite{7078985}, audio \cite{flanagan} to mention just a few.
For more details we refer to our recent survey article \cite{MR4094471}.

\subsection{Problem formulation and related results}
The object which is in the spot light of this article is the short time Fourier transform with Gaussian window as introduced next.
\begin{defn}
 The \emph{Gabor transform} of $f\in L^2(\R)$ is defined by 
 \begin{equation}
  \G f(x,y):=\left(f(\cdot) e^{-\pi(\cdot - x)^2} \right)^\wedge(y),\quad x,y\in\R.
 \end{equation}
Moreover, we denote the \emph{spectrogram} by $\Sp f:=|\G f|^2$.
\end{defn}
As it is well known, the Gabor transform essentially produces entire functions. 
Its holomorphicity makes $\G f$ a rather alluring object to study, as one can leverage this strong structural feature and draw upon tools from complex function theory  to analyze it.
For example, it follows from the open mapping theorem that $f\in L^2(\R)$ is uniquely determined (up to a unimodular constant) by the phase-less information of its spectrogram $\Sp f$ on any open set which can be arbitrarily small.
The prior information that the function to be reconstructed is holomorphic has proved to be fruitful in a number of phase retrieval problems. 
To name just a few we mention \cite{MR80802,1053648,doi:10.1080/713817747,MR3256781,MR3421917}.\\

When dealing with real world measurement scenarios stability is absolutely essential, due to the occurrence of measurement errors, imprecision of machine arithmetic or imperfect modelling.
In the context of phase retrieval, stability is typically formalized as some sort of Lipschitz estimate.
For the sake of concreteness, in our case
\begin{equation}
    \min_{\tau \in \T} \|g-\tau f\|_{L^2(\R)}\le C \||\G g| -|\G f|\|_{L^2(\R^2)},\footnote{We stress that the choice of norms here are completely arbitrary and any other pair of similarity measures would make just as much sense.}
\end{equation}
with $C$ an absolute constant independent from $f,g\in L^2(\R)$ would express uniform stability, as any pair of functions possessing similar spectrograms must be similar themselves.
Unfortunately, the situation is more complicated than that. By choosing two (or more) components $f_1,f_2$ which have their respective Gabor transforms concentrated on disjoint sets in the plane, meaning that their product $\G f_1 \cdot \G f_2$ essentially vanishes, one routinely produces instabilities by setting $f=f_1+f_2$ and $g=f_1-f_2$ which have very similar spectrograms while $f$ and $g$ are not similar at all in the above sense \cite{MR3989716}. 
This phenomenon is intrinsically manifested in any phase retrieval problem in an infinite dimensional setting. In particular, there is no uniform stability in infinite dimensions \cite{MR3554699,MR3656501}.\\

To take this into account, one could adjust the formulation and ask for local stability in the sense that given a domain $\Omega\subseteq\R^2$ and a fixed signal $f\in L^2(\R)$, what can be said about the optimal constant satisfying
\begin{equation}
    \min_{\tau\in\T}\|\G g- \tau \G f\|_{L^2(\Omega)}\le C(f,\Omega) \||\G g| - |\G f|\|_{L^2(\Omega)},\quad \forall g\in L^2(\R)?
\end{equation}
This notion is local in a twofold meaning. Firstly, local in signal $f$ under consideration, and secondly local in the domain $\Omega$ where the phase information is to be recovered.\\

In previous work \cite{MR3935477,jemsgrohsrathmair} we established that in the case of Gabor phase retrieval there is no other type of instability than the multi-component type as described above.
\begin{thm}[\cite{MR3935477,jemsgrohsrathmair}, informal]\label{thm:stabilitypriorwork}
Let $\Omega\subseteq \R^2$ be a domain and $f\in L^2(\R)$.
Then it holds for all $g\in L^2(\R)$ that 
\begin{equation}\label{eq:stabilityprior}
    \min_{\tau\in\torus} \|\G g-\tau \G f\|_{L^1(\Omega)} \lesssim h(f,\Omega)^{-1} \cdot \||\G g|-|\G f|\|_{\mathcal{W}(\Omega)},
\end{equation}
where $\|\cdot\|_{\mathcal{W}(\Omega)}$ denotes a first order Sobolev norm with a polynomial weight, and where $h(f,\Omega)$ is the \emph{Cheeger constant}
defined by
\begin{equation}
    h(f,\Omega):= \inf_{E\subseteq \Omega} \frac{ \int\limits_{\Omega\cap \partial E} |\G f| }{\min\left\{\int_E |\G f|, \int_{\Omega\setminus E} |\G f| \right\}}
\end{equation}
\end{thm}
The Cheeger constant measures connectedness in the transform domain. Morally, $h(f,\Omega)$ is small if and only if
there exists a partition of $\Omega=E \cup (\Omega\setminus E)$ such that the overall energy is evenly distributed between the two components, and furthermore $|\G f|$ is fairly small on the separating boundary.
Whereas, if $f$ is single-component -- which loosely speaking means it is not multi-component -- the ratio appearing in the Cheeger constant cannot become too small, and hence its reciprocal not too large.\\
Recently Cheng \etal \cite{cheng2020stable} have considered a rather abstract and general setting for a phase retrieval problem, under the assumption that the global task can be broken up into a bunch of uniformly stable sub problems.
To such a scheme is naturally associated a graph whose vertices correspond to the individual sub problems. The weighting on the edges indicates to what extent the sub problems are connected to each other.
Their main result reveals a connection between the stability of the original global problem in terms of the connectivity of the associated graph, i.e. its graph Cheeger constant.

\subsection{This paper}
The work presented in this paper has been driven by the desire to improve upon Theorem \ref{thm:stabilitypriorwork}.
While characterizing the stability of the phase reconstruction process by the connectivity of the measurements is a very satisfactory result the choice of the norm on the right hand side in \eqref{eq:stabilityprior} is unpleasant.
In order to guarantee stable recovery, one would need to be able to record $|\G f|$ and its derivatives on a continuum $\Omega.$\\
It is the main objective of this paper to develop new proof techniques in order to alleviate the norm on the right hand side and therefore obtain stronger stability results for Gabor phase retrieval.
The underlying strategy is very much inspired by the work of Cheng \etal \cite{cheng2020stable} as our approach as well hinges on  identifying local sub problems whose stability we can understand and then stitch these together in order to obtain more general stability statements. 

\subsubsection{Unit square as the base case}
Initially, we will lay our focus on studying stability on the arguably simplest domain one can think of, a square of side length $1$.
For this case, we will establish the following stability result in Section \ref{sec:stabilitysquare}.
\begin{thmlet}\label{thm:specialcasebasicgabor}
 Let $Q\subseteq\R^2$ be a square of side length $1$. Then it holds for all $f,g\in L^2(\R)$ that 
 \begin{equation}\label{eq:specialcasebasicgabor}
  \min_{\tau\in\torus} \|\G g- \tau \G f\|_{L^2(Q)} \lesssim \left(\frac{\|\Sp f\|_{L^\infty(\R^2)} + \|\Sp g\|_{L^\infty(\R^2)}}{\|\Sp f\|_{L^1(Q)}}\right)^{1/2}\cdot \|\Sp f -\Sp g\|_{L^2(Q)}^{1/2}.
 \end{equation}
\end{thmlet}
The aforementioned result implies that, once we demand that the signals we consider satisfy for example that $\|\Sp f\|_{L^\infty}, \|\Sp g\|_{L^\infty}\le 1$, we have stability on any unit square which is located in a position in the time-frequency plane where $f$ has significant portion of its overall energy. 
That the stability potentially behaves badly when moving to squares which feature very little energy, may be considered acceptable as there is not too much to reconstruct anyways in that case.
In fact, the stability can be arbitrarily bad if the square is located in an unfavourable position.
To convince the reader that Theorem \ref{thm:specialcasebasicgabor} is a somewhat sharp result we have a closer look at a concrete example.
\begin{ex}\label{ex:sharpness}
Let $Q$ denote the square of side length $1$ centered at the origin, with its sides parallel to the axes and let $a>0$.
With $\varphi=e^{-\pi\cdot^2}$ denoting the Gaussian, we define
\begin{equation}
    f=f_a:=\varphi(\cdot+a)+\varphi(\cdot-a),\quad\text{and}\quad g=g_a:=\varphi(\cdot+a)-\varphi(\cdot-a).
\end{equation}
According to Lemma \ref{lem:estexsharp}, the ratio between the distance of the signals and the distance of the measurements deteriorates rather quickly as $a\rightarrow \infty$, i.e., 
\begin{equation}
    \frac{\min_{\tau\in\torus} \|\G g-\tau \G f\|_{L^2(Q)}}{\|\Sp f - \Sp g\|_{L^2(Q)}^{1/2} } \gtrsim e^{\pi a}.
\end{equation}
In our stability estimate \eqref{eq:specialcasebasicgabor} the term $\|\Sp f\|_{L^1(Q)}$ appearing in the denominator compensates for this type of degeneration.\\
In Lemma \ref{lem:estexsharp} we also compute the respective Gabor transforms of $f$ and $g$, and find that 
\begin{equation}
    \G f(x,y) = \exp\left\{-\frac\pi2 a^2 -\frac\pi2 |z|^2 -\pi\mi xy \right\} \cos(a\pi\mi \bar{z}), \quad z=x+\mi y.
\end{equation}
Based on this factorization we make a couple of observations: As $a\rightarrow \infty$ 
\begin{enumerate}[i)]
    \item the energy $\|\Sp f\|_{L^1(Q)}$ on the square under consideration decreases (exponentially fast),
    \item the roots of the anti-holomorphic factor $\cos(a\pi \mi \bar{z})$ (and thus also the roots of $\G f$) accumulate at a density of $\sim a^{-1}$ on the imaginary axis,
    while its modulus grows like $e^{a\pi|x|}$; 
    at the same time, $\G f$ becomes strongly oscillatory in the $y$ direction, and thus sustains a loss of smoothness.
\end{enumerate}
Let us put this into perspective with Theorem \ref{thm:specialcasebasicgabor}.
As discussed before, Theorem \ref{thm:specialcasebasicgabor} tells us for $f$ to be an instability the square under consideration must feature little of the overall energy. In that sense, observation (i) is a necessary property.\\
It would be interesting to know if the same can be said about observation (ii), i.e., for $f$ to be an instability, is it necessary that $\mathcal{G}f$ possesses many roots along a line passing through $Q$?
\end{ex}

The proof methods employed to establish Theorem \ref{thm:specialcasebasicgabor} are fundamentally different to the techniques used in our prior work \cite{MR3935477,jemsgrohsrathmair}.
The proof of Theorem \ref{thm:specialcasebasicgabor} relies on the idea that the recovery of the missing phase information can be formulated as an extension problem: 
If we only know certain values of the modulus of a holomorphic function $F$, this means that we only have access to values of the tensor $F\otimes\bar{F}(z,\zeta)=F(z)\overline{F(\zeta)}$ on a subset of the diagonal $\zeta=z$.
Phase information can be obtained by inferring values of the tensor beyond the diagonal. More precisely, we will make use of the expansion formula 
\begin{equation}\label{eq:recformulatensor}
    F\otimes\bar{F}(z,\zeta) = \sum_{k,l\in\N} \frac{(\dbar^l \dd^k |F|^2)(0)}{k!l!}z^k\bar{\zeta}^l,
\end{equation}
which provides an explicit formula for our reconstruction task $\Sp f\rightsquigarrow \G f\sim e^{i\alpha}\G f$.
Clearly, in order to obtain a quantitative result from formula \eqref{eq:recformulatensor} we have to analyze the family of differential operators $(\dbar^l\dd^k)_{k,l\in\N}$ acting on the squared modulus of certain analytic functions.
If these functions are sufficiently smooth locally, we will be able to derive useful bounds for the differential operators. 
For entire functions the distribution of zeros -- and thus local smoothness -- is strongly connected to the global growth.
In fact as we shall see, if growth can be controlled, which can be done in our case, so can smoothness.
\subsubsection{Stability on larger domains}
Having established stability on unit squares we will move on to analyzing stability on more general, larger domains. 
Before we present our main result into this direction we introduce the concepts we require.

\begin{defn}
A weighted graph consists of a graph $\graphg=(\graphv,\graphe)$ with vertex set $\graphv$ and edge set $\graphe$, and a pair of functions $w:\graphv\rightarrow \R_+$ and $\sigma:\graphe\rightarrow\R_+$.
By setting $\sigma(e)=0$ if $e\in \graphv\times\graphv \setminus \graphe$ we may assume that $\sigma$ is defined on all of $\graphv\times\graphv$.
\end{defn}
To capture the connectedness of a weighted graph we introduce the graph version of the Cheeger constant.
\begin{defn}
Let $(\graphg,w,\sigma)$ be a weighted graph.
The graph \emph{Cheeger constant} of $\graphg$ is defined by 
\begin{equation}
    h_\graphg := \inf_{S\subseteq \graphv} \frac{\sigma(\partial S)}{\min\left\{w(S),w(\graphv\setminus S) \right\}},
\end{equation}
where the boundary $\partial S$ denotes the set of all edges which have one endpoint in $S$ and the other one in the complement $\graphv\setminus S$, and where for a set $V\subseteq \graphv$ we write $w(V)$ for $\sum_{v\in V} w(V)$, and likewise for 
$ E\subseteq \graphe$ we write $w(E)$ instead of $\sum_{e\in\graphe} \sigma(e)$.
\end{defn}
Our general stability result, which will be established in Section \ref{sec:stabilitygeneral}, is very much reminiscent of Theorem \ref{thm:stabilitypriorwork}.
\begin{thmlet}[Theorem \ref{thm:generalcasecheeger}, simplified]\label{thm:b}
Let $f\in L^2(\R)$ and let $\mathcal{Q}$ be a finite family of unit squares.
We set $\Omega:= \bigcup_{Q\in\mathcal{Q}} Q,$ and define a weighted graph by identifying the vertex set with $\mathcal{Q}$, and by setting 
\begin{equation}
    w(v):=\|\Sp f\|_{L^1(Q)}, v\in\graphv \quad \text{and} \quad \sigma(e):= \|\Sp f\|_{L^1(Q\cap Q')}^2, \,e=(v,v')\in\graphe.\footnote{We implicitly identify $Q$ with $v$ and $Q'$ with $v'$}
\end{equation}
Then it holds for all $g\in L^2(\R)$ that 
\begin{equation}\label{eq:estthmb}
    \min_{\tau\in\torus} \|\G g- \tau \G f\|_{L^2(\Omega)} \lesssim h_\graphg^{-2} \cdot \|\Sp f - \Sp g\|_{L^2(\Omega)}^{1/2}, 
\end{equation}
with the implicit constant mildly depending on $\Sp f, \Sp g$ and $\mathcal{Q}$.
\end{thmlet}

\begin{rem}
At this stage we would like to stress that Theorem \ref{thm:b} does not represent an unrestricted improvement on Theorem \ref{thm:stabilitypriorwork}. This may not be instantly apparent from the simplified version above and require a closer look at Theorem \ref{thm:generalcasecheeger} and its proof.
The main limitation of our new results consists in the restriction to finite domains. 
To give a concrete example, our prior results in \cite{MR3935477,jemsgrohsrathmair} produce a finite stability constant for the signal $f$ being the Gaussian, and the domain being the full time-frequency plane, whereas Theorem \ref{thm:b} fails this case. 
\end{rem}

The proof techniques employed to establish the above result are strongly inspired by the work of Cheng \etal \cite{cheng2020stable}.
Since, the sub problems we consider are not uniformly stable, we need to make some minor adaptations.
However, once the graph is correctly set up, our results follow from essentially replicating the proofs in the aforementioned article.

\subsubsection{Discretization}
In Section \ref{sec:discretization} we aim to take the stability analysis one step further and consider the case where one can only access the spectrogram data on a discrete set of sampling points. 
Into this direction, in Section \ref{sec:discretization} we will establish a result of the following flavour.
\begin{thmlet}[Theorem \ref{thm:discretizationsquare}, simplified]
Assume the conditions of Theorem \ref{thm:b}. Moreover, suppose $\Omega\subseteq \R^2$ is a square of side length $s>0$.
Then, for all $\varepsilon>0$,
there exists a finite set $\Lambda \subseteq \Omega$ with $|\Lambda| \sim \max\left\{\ln(1/\varepsilon),s\right\}^2$, with associated weights $(w_\lambda)_{\lambda\in\Lambda}\subseteq \R_+$ satisfying that $\sum_\lambda w_\lambda=s^2$ such that it holds for all $g\in L^2(\R)$ that
\begin{equation}
    \min_{\tau\in\torus} \|\G g- \tau \G f\|_{L^2(\Omega)} \lesssim h_\graphg^{-2} \left( \|\Sp f - \Sp g\|_{\ell^2_w(\Lambda_\varepsilon)}^{1/2} +\varepsilon \right),
\end{equation}
with the implicit constant mildly depending on $\Sp f, \Sp g$ and $\mathcal{Q}$.
\end{thmlet}
The above result will be obtained by combining Theorem \ref{thm:b} with a tailor-made error estimate for a carefully picked cubature rule.
Note that, if we write out the norm on the right hand side in \eqref{eq:estthmb}, we get an integral over a real analytic function. 
Standard estimates for numerical integration schemes typically relate the error to the smoothness of the function to be integrated. 
We shall pursue a different path here, and show that the error can be controlled in terms of the size of the holomorphic extension of the integrand, cf. Theorem \ref{thm:chawlaadapted}.
In our case, this holomorphic extension is quite easy to handle. 
This approach saves us from estimating high order derivatives of the integrand $(\Sp f - \Sp g)^2$, which would have been a rather tedious mission.
\subsection{Notation}
Throughout we will identify the complex plane $\C$ with $\R^2$ by virtue of $z\simeq x+\mi y$. 
We use \emph{Wirtinger} derivatives defined by 
\begin{equation}
 \dd=\frac12 \left(\frac{\dd}{\dd x}-\mi \frac{\dd}{\dd y}\right) \quad\text{and}\quad 
 \dbar = \frac12 \left(\frac{\dd}{\dd x}+\mi \frac{\dd}{\dd y}\right).
\end{equation}
The two dimensional Lebesgue measure will be denoted by $A$. 
We write  $B_r=\{(x,y):\, \|(x,y)\|_2 < r\}$ for the open disk of radius $r>0$ centered at the origin. 
If the center is different from the origin, say at $z_0$ we write $B_r(z_0)$.
%%%%%%%%%%%%%%%%%%%%%%%%%%%%%%%%%%%%%%%%%%%%%%%%%%%%%%%%%%%%%%%%%%%%%%%%%%%%%%%%%%%%%%
\section{Fully continuous measurements}\label{sec:fullycontinuous}
This section is devoted to establishing $L^2$-stability estimates for Gabor phase retrieval.
As outlined in the introduction at the beginning we consider very basic domains.
A more general result will be obtained by stitching together these local results in a suitable manner.
First and foremost, we begin with collecting and recalling a couple of facts which will turn out to be useful for our further analysis.
\subsection{Prerequisites}
As already mentioned, our setting essentially takes place in a space of analytic functions.
We will make this more concrete now.
\begin{defn}
The \emph{Fock space} $\F^p$, $1\le p \le \infty$ consists of all entire functions $F$, such that 
\begin{equation}
    \|F\|_{\F^p}:=
     \left\| F \exp\left\{-\frac\pi2 |\cdot|^2\right\}\right\|_{L^p(\C)} < +\infty.
\end{equation}
\end{defn}
The Fock spaces form a nested family, that is, if $p<p'$ it holds that $\F^p \subseteq \F^{p'}$.
The space $\F^2$ will be particularly interesting for our purpose.
Equipped with the canonical inner product 
  \begin{equation}
   \langle F,G\rangle_{\F^2} = \int_\C F(z)\overline{G(z)} e^{-\pi |z|^2}\,dA(z).
  \end{equation}
  $\F^2$ is a Hilbert space. In fact, it is a reproducing kernel Hilbert space and the reproducing kernel is given by $e^{\pi z \bar{w}}$, meaning that for all $F\in\F^2$ we have that
\begin{equation}
 F(z)=\int_\C F(w) e^{\pi z\bar{w}} e^{-\pi|w|^2} \,dA(w),\quad  z\in\C.
\end{equation}
There is a natural identification between $L^2(\R)$ and $\F^2$.
More precisely, the so-called \emph{Bargmann} transform defined by 
\begin{equation}
 \mathfrak{B}f(z):= \int_\R f(t) \exp\left\{2\pi t z- \pi t^2 -\frac{\pi}2 z^2\right\} \,dt
\end{equation}
is an isometry from $L^2(\R)$ to $\F^2$, and moreover, for all $f\in L^2(\R)$ it holds that 
\begin{equation}
 \G f(x,-y) = \mathfrak{B}f(z) \exp\left\{\pi \mi xy -\frac{\pi}2 |z|^2 \right\}, \qquad z=x+\mi y\in\C.
\end{equation}

\subsection{Basic case}\label{sec:stabilitysquare}
This section is devoted to establishing a first stability result, which formulated in terms of analytic functions in the Fock space reads as follows.
\begin{thm}\label{thm:basicstabfock}
There exists a universal constant $K>0$ such that 
for all $r>0$ and $F,G\in \F^2$ it holds that 
 \begin{equation}
  \min_{\tau\in\torus} \left\| G-\tau F\right\|_{L^2(B_r)} \le K  r^2 e^{4\pi^2 r^2}  \frac{\|F\|_{\F^\infty} + \|G\|_{\F^\infty}}{\|F\|_{L^2(B_r)}} \||F|^2 - |G|^2\|_{L^2((-1/2,1/2)^2)}^{1/2}.
 \end{equation}
\end{thm}
From the above result it follows immediately that for all $f,g\in L^2(\R)$ and $r>0$ it holds that 
\begin{multline}\label{eq:estlocalgaborbase}
    \min_{\tau\in\torus} \left\| \G g-\tau \G f\right\|_{L^2(B_r)} \\
    \lesssim  r^2 e^{4\pi^2 r^2} \left(\frac{\|\Sp f\|_{L^\infty(\C)} + \|\Sp g\|_{L^\infty(\C)}}{\|\Sp f\|_{L^1(B_r)}} \right)^{1/2}  \|\Sp f - \Sp g\|_{L^2((-1/2,1/2)^2)}^{1/2}.
\end{multline}
What if we are given spectrogram data on a unit square which is different from $(-1/2,1/2)^2$ though?
That is, what if the sides of the square are not parallel to the axes, and/or the center is different from the origin?
To deal with this general case we make a simple observation.
\begin{rem}\label{rem:usegeneralsquare}
Note that, the Fock space $\mathcal{F}^2$ is invariant w.r.t. rotations, i.e. with $\theta \in [0,2\pi)$ it holds that $F(e^{\mi\theta}\cdot) \in \mathcal{F}^2$ if and only if $F \in \mathcal{F}^2.$
Even though $\mathcal{F}^2$ cannot just be simply replaced by $\mathcal{G}L^2(\R)$ in this statement, it holds true that there exists a real-valued function $\mu_\theta$ such that for all $f\in L^2(\R)$ there exists $f_\theta\in L^2(\R)$ satisfying that
\begin{equation}
    \G f(e^{\mi\theta}z) = \G f_\theta(z) \cdot e^{\mi \mu_\theta(z)}.
\end{equation}
The situation is similar with translations. That is, given $\tau\in \mathbb{C}$ there exists a real-valued function $\nu_\tau$ such that for all $f \in L^2(\R)$ there exists $f_\tau \in L^2(\R)$ satisfying that 
\begin{equation}
    \G f(z - \tau) = \G f_\tau(z) \cdot e^{\mi \nu_\tau(z)} .
\end{equation}
\end{rem}
Remark \ref{rem:usegeneralsquare} allows us to replace $(-1/2,1/2)^2$ by an arbitrary square of unit size in \eqref{eq:estlocalgaborbase} and we obtain 
\begin{cor}\label{cor:basicstabgabor}
Let $r>0$, $z_0\in\C$, $f,g\in L^2(\R)$ and let $Q$ be a square of unit size with center at $z_0$. There exists a universal constant $K>0$ (independent of $r,z_0,f,g,Q$) such that 
 \begin{multline}
  \min_{\tau\in\torus} \|\G g- \tau \G f\|_{L^2(B_r(z_0))} \\
  \le K   r^2 e^{4\pi^2 r^2}  \left(\frac{\|\Sp f\|_{L^\infty(\C)} + \|\Sp g\|_{L^\infty(\C)}}{\|\Sp f\|_{L^1(B_r(z_0))}} \right)^{1/2} \|\Sp f  - \Sp g\|_{L^2(Q)}^{1/2}.
 \end{multline}
\end{cor}

The proof of Theorem \ref{thm:basicstabfock} is split up into several sub steps.
At the heart of the proof lies a formula that allows us to recover the tensor $F\otimes\bar{F}$ from $|F|$.
Given two functions $F,G\in L^2(B_r)$ we introduce the quantity
\begin{equation}\label{def:deltar}
\delta_r(F,G):= \left\|F\otimes\bar{F}-G\otimes\bar{G} \right\|_{L^2(B_r\times B_r)},
\end{equation}
which will serve as an intermediate stepping stone in our stability analysis.
\begin{prop}\label{prop:recformula}
 For all $F\in \mathcal{O}(\C)$ it holds that 
 \begin{equation}
  F\otimes \bar{F} (z,\zeta)= \sum_{k,l\in\N} \frac{\dbar^l \dd^k |F|^2(0)}{k!\cdot l!} z^k\bar{\zeta}^l,\quad z,\zeta\in\C.
 \end{equation}
 Moreover, for $r>0$ if also $G\in \mathcal{O}(\C)$ it holds that 
 \begin{equation}
  \delta_r(F,G)^2
  = \sum_{k,l\in \N} \omega_k \omega_l \left|\dbar^l\dd^k (|F|^2-|G|^2)(0) \right|^2,
 \end{equation}
 where the weights are defined by
 \begin{equation}\label{def:gamma}
  \omega_k=\omega_k(r):=\frac{\|z^k\|_{L^2(B_r)}^2}{(k!)^2} = \frac{\pi r^{2k+2}}{k!(k+1)!}.
 \end{equation}
\end{prop}
\begin{proof}
 Note that $F$ is holomorphic, and therefore $\bar{F}$ is antiholomorphic.
 Thus, we get that 
 \begin{equation}
  \dbar^l \dd^k |F|^2 = \dbar^l \dd^k (F\cdot\bar{F}) = F^{(k)} \cdot \overline{F^{(l)}}.
 \end{equation}
Now, Taylor expansion yields the first statement since
\begin{equation}
 F\otimes\bar{F}(z,\zeta) = F(z)\cdot \overline{F(\zeta)} = \left(\sum_{k\in\N} \frac{F^{(k)}(0)}{k!}z^k \right)\cdot \overline{\left(\sum_{l\in\N} \frac{F^{(l)}(0)}{l!}\zeta^l\right)}.
\end{equation}
To prove the second statement we set $u:=|F|^2-|G|^2$ and observe that $\left(z^k\bar{\zeta}^l\right)_{k,l}$ forms a system of pairwise orthogonal vectors in $L^2(B_r\times B_r)$.
Hence, it follows from the first part that 
\begin{equation}
\begin{aligned}
  \delta_r(F,G)^2
  &= \sum_{k,l\in \N} \left|\dbar^l\dd^k u(0) \right|^2 \cdot \frac{\left\|z^k\bar{\zeta}^l\right\|_{L^2(B_r\times B_r)}^2}{(k!)^2\cdot (l!)^2}\\
  &= \sum_{k,l\in \N} \left|\dbar^l\dd^k u(0) \right|^2 \cdot 
  \frac{\|z^k\|_{L^2(B_r)}^2}{(k!)^2} \cdot \frac{\|\zeta^l\|_{L^2(B_r)}^2}{(l!)^2}\\
  &=  \sum_{k,l\in \N} \left|\dbar^l\dd^k u(0) \right|^2 \omega_k \omega_l.
\end{aligned}
\end{equation}
\end{proof}
In the next step we relate $\delta_r(F,G)$ to a more familiar notion of similarity.
\begin{prop}\label{prop:reldeltadist}
 For all $r>0$ and $F,G\in L^2(B_r)$ it holds that 
 \begin{equation}
  \min_{\tau\in\torus} \left\|G-\tau F\right\|_{L^2(B_r)}^2 \lesssim \frac{\delta_r(F,G)^2}{\|F\|_{L^2(B_r)}^2},
 \end{equation}
with the minimal implicit constant being not greater than $5$.
\end{prop}
\begin{proof}
To simplify notation we omit the explicit reference to the $L^2(B_r)$ norm and innner product, respectively and define three numbers
\begin{equation}
 \alpha:=\|F\|,\quad \beta:=\|G\|,\quad \gamma:=|\langle F,G\rangle|.
\end{equation}
Note that Cauchy-Schwarz inequality tells us that $\gamma\le \alpha\beta$.
Provided that $\gamma\neq 0$ we can estimate
\begin{equation}
 \min_{\tau\in\torus} \left\|G-\tau F\right\|^2 \le \left\|G-\frac{\langle G,F\rangle}{\gamma}F\right\|^2 = \beta^2+\alpha^2-2\gamma = (\alpha-\beta)^2 + 2(\alpha\beta-\gamma).
\end{equation}
Obviously, the inequality remains true if $\gamma=0$.
We rewrite
\begin{equation}
  \delta_r(F,G)^2 = \alpha^4 + \beta^4-2\gamma^2 = (\alpha+\beta)^2(\alpha-\beta)^2+2(\alpha\beta+\gamma)(\alpha\beta-\gamma).
\end{equation}
Thus, it suffice to show that
\begin{equation}\label{eq:tobeestd}
 \alpha^2 \frac{(\alpha-\beta)^2+2(\alpha\beta-\gamma)}{(\alpha+\beta)^2(\alpha-\beta)^2+2(\alpha\beta+\gamma)(\alpha\beta-\gamma)}\le 5.
\end{equation}
Since we have that the first part of the left hand side
\begin{equation}
 \alpha^2 \frac{(\alpha-\beta)^2}{(\alpha+\beta)^2(\alpha-\beta)^2+2(\alpha\beta+\gamma)(\alpha\beta-\gamma)}\le\frac{\alpha^2}{(\alpha+\beta)^2}\le 1,
\end{equation}
it remains to show that 
\begin{equation}\label{eq:tobeestd2}
 \frac{2\alpha^2(\alpha\beta-\gamma)}{(\alpha+\beta)^2(\alpha-\beta)^2+2(\alpha\beta+\gamma)(\alpha\beta-\gamma)}\le 4.
\end{equation}
To prove \eqref{eq:tobeestd2} we distinguish two cases depending on a) $\beta>\alpha/2$ and b) $\beta \le \alpha/2$. \\

In case a) we estimate
\begin{equation}
 \begin{aligned}
  \frac{2\alpha^2(\alpha\beta-\gamma)}{(\alpha+\beta)^2(\alpha-\beta)^2+2(\alpha\beta+\gamma)(\alpha\beta-\gamma)}
  &\le \frac{2\alpha^2(\alpha\beta-\gamma)}{2(\alpha\beta+\gamma)(\alpha\beta-\gamma)}\\
  &\le\frac{\alpha^2}{\alpha\beta+\gamma}\\
  &\le \frac{\alpha^2}{\alpha^2/2} \\
  &=2.
 \end{aligned}
\end{equation}
In case b) we have that 
\begin{equation}
 \begin{aligned}
  \frac{2\alpha^2(\alpha\beta-\gamma)}{(\alpha+\beta)^2(\alpha-\beta)^2+2(\alpha\beta+\gamma)(\alpha\beta-\gamma)} &\le \frac{\alpha^4}{(\alpha+\beta)^2(\alpha-\beta)^2}\\
  &\le \frac{\alpha^4}{\alpha^4/4}\\
  &=4,
 \end{aligned}
\end{equation}
which settles the proof.
\end{proof}

The upcoming Proposition deals with estimating $\delta_r(F,G)$.
\begin{prop}\label{prop:bounddelta}
For all $F,G\in \F^\infty$ and $r>0$ it holds that 
 \begin{equation}
  \delta_r(F,G)^2\lesssim r^4 e^{8\pi^2 r^2}\cdot \left(\|F\|_{\F^\infty}^2 + \|G\|_{\F^\infty}^2 \right) \cdot \||F|^2-|G|^2\|_{L^2(Q)}.
 \end{equation}

 \end{prop}
In order to prove Proposition \ref{prop:bounddelta} we require two auxiliary results. 
The first one tells us that the derivatives of a small function can only be large if the function is not too smooth.
\begin{lem}\label{lem:boundderivsquare}
 Suppose that $u\in C^\infty(Q)$, where $Q=[-1/2,1/2]^2$.
 Then it holds for all $k,l\in\N$ that 
 \begin{multline}
  \left| \dbar^l \dd^k u(0)\right|^2 \lesssim \|u\|_{L^2(Q)} \\
  \times 
  \left(\|\dbar^{2l}\dd^{2k}u\|_{L^2(Q)} + \|\dbar^{2l}\dd^{2k+2}u\|_{L^2(Q)} +\|\dbar^{2l+2}\dd^{2k}u\|_{L^2(Q)}  \right).
 \end{multline}
\end{lem}
\begin{proof}
First note that it follows from the fundamental theorem of calculus that 
for every function $v\in C^\infty(Q)$ it holds that 
\begin{equation}
 |v(0)|^2 \lesssim \|v\|_{L^2(Q)}^2 + \|\nabla v\|_{L^2(Q)}^2.
\end{equation}
Plug in $v=\dbar^l\dd^k u$ to obtain that 
\begin{equation}\label{eq:estorl2norm}
 |\dbar^l\dd^k u(0)|^2 \lesssim \|\dbar^l \dd^k u\|_{L^2(Q)}^2 + \|\dbar^l \dd^{k+1}u\|_{L^2(Q)}^2 + \|\dbar^{l+1}\dd^k u\|_{L^2(Q)}^2.
\end{equation}
 We expand $u$ in its Fourier series 
 \begin{equation}
  u(x,y)=\sum_{m,n\in\Z} \hat{u}_{m,n} e^{2\pi\mi(mx+ny)},
 \end{equation}
and observe that 
\begin{equation}
\begin{aligned}
 \dd e^{2\pi\mi(mx+ny)} &= \pi\mi (m-\mi n) e^{2\pi\mi(mx+ny)},\\
 \dbar e^{2\pi\mi(mx+ny)} &= \pi\mi(m+\mi n) e^{2\pi\mi(mx+ny)},
\end{aligned}
\end{equation}
and furthermore that for $p,q\in\N$
\begin{equation}
    \dbar^q\dd^p u(x,y)=\sum_{m,n\in\Z} \hat{u}_{m,n}(\pi\mi)^{p+q}(m-\mi n)^p(m+\mi n)^q e^{2\pi\mi(mx+ny)}.
\end{equation}
 Thus, using Cauchy-Schwarz inequality allows us to bound 
 \begin{equation}
  \begin{aligned}
   \left\|\dbar^q\dd^p u\right\|_{L^2(Q)}^2 
        &= \sum_{m,n\in\Z} |\hat{u}_{m,n}|^2 \pi^{2(p+q)} (m^2+n^2)^{p+q}\\
        &\le \left(\sum_{m,n\in\Z} |\hat{u}_{m,n}|^2 \right)^{1/2}
        \cdot \left(\sum_{m,n\in\Z} |\hat{u}_{m,n}|^2 \pi^{4(p+q)}(m^2+n^2)^{2(p+q)} \right)^{1/2}\\
        &= \|u\|_{L^2(Q)} \cdot \|\dbar^{2q}\dd^{2p}u\|_{L^2(Q)}.
  \end{aligned}
 \end{equation}
Combining the estimate with \eqref{eq:estorl2norm} yields the claim.
\end{proof}
Next we relate local smoothness of the functions under consideration to their global growth. 
\begin{lem}\label{lem:smoothgrowth}
 Let $F\in \F^2$ and $p\in \N$. 
 With $Q=[-1/2,1/2]^2$ it holds that 
 \begin{equation}
  \|F^{(p)}\|_{L^\infty(Q)} \le 2^{p+3} \pi^{p+1} \Gamma\left(\frac{p}2 +1\right)\cdot \|F\|_{\F^\infty}.
 \end{equation}
\end{lem}
\begin{proof}
 We make use of the reproducing property of $\F^2$ which implies that 
 \begin{equation}
  F^{(p)}(z) = \int_\C F(w) (\pi\bar{w})^p e^{\pi z\bar{w}} e^{-\pi|w|^2}\,dA(w).
 \end{equation}
With this we estimate
\begin{equation}
  \left|F^{(p)}(z) \right| \le \pi^p \cdot \|F\|_{\F^\infty} \cdot 
  \int_\C |w|^p \exp\left\{-\frac{\pi}{2}|w|^2 +\pi \Re\{z\bar{w}\} \right\}\,dA(w)
\end{equation}
Since $z\in Q$ implies that $|z|\le 1/\sqrt{2}$ we can further bound the integral by
\begin{equation}
  \int_\C |w|^p \exp\left\{-\frac{\pi}{2}|w|^2 + \frac{\pi}{\sqrt{2}} |w|\right\}\,dA(w)
        =2\pi\int\limits_0^\infty r^{p+1} \exp\left\{-\frac{\pi}{2} r^2 + \frac{\pi}{\sqrt{2}} r\right\} \,dr,
\end{equation}
which -- according to Lemma \ref{lem:estgammalikeint} -- is bounded by $2\pi \cdot 2^{p+2} \Gamma\left(\frac{p}2+1\right)$.
\end{proof}
We are now well prepared for the
\begin{proof}[Proof of Proposition \ref{prop:bounddelta}]
 We set $u:=|F|^2-|G|^2$.
 We combine Proposition \ref{prop:recformula} and Lemma \ref{lem:boundderivsquare} to obtain that
 \begin{multline}
  \delta_r(F,G)^2 
  \lesssim \|u\|_{L^2(Q)} \\
  \times \sum_{k,l\in\N} \omega_k\omega_l \left(\|\dbar^{2l}\dd^{2k}u\|_{L^2(Q)} + \|\dbar^{2l}\dd^{2k+2}u\|_{L^2(Q)} +\|\dbar^{2l+2}\dd^{2k}u\|_{L^2(Q)}  \right).
 \end{multline}
It follows from Lemma \ref{lem:smoothgrowth} that 
\begin{equation}
 \begin{aligned}
  \|\dbar^{2l}\dd^{2k}u\|_{L^2(Q)} 
        &\lesssim \|\dbar^{2l}\dd^{2k}|F|^2\|_{L^2(Q)} + \|\dbar^{2l}\dd^{2k}|G|^2\|_{L^2(Q)}\\
        &\lesssim \|F^{(2k)}\|_{L^\infty(Q)} \|F^{(2l)}\|_{L^\infty(Q)} + \|G^{(2k)}\|_{L^\infty(Q)} \|G^{(2l)}\|_{L^\infty(Q)}\\
        &\lesssim \left(\|F\|_{\F^\infty}^2 + \|G\|_{\F^\infty}^2\right)\cdot 
        (2\pi)^{2k+2l} \cdot k!\cdot l! 
 \end{aligned}
\end{equation}
Likewise, one shows that
\begin{equation}
 \|\dbar^{2l}\dd^{2k+2}u\|_{L^2(Q)} \lesssim \left(\|F\|_{\F^\infty}^2 + \|G\|_{\F^\infty}^2\right)\cdot (2\pi)^{2k+2l} \cdot (k+1)!\cdot l!,
\end{equation}
as well as 
\begin{equation}
 \|\dbar^{2l+2}\dd^{2k}u\|_{L^2(Q)} \lesssim \left(\|F\|_{\F^\infty}^2 + \|G\|_{\F^\infty}^2\right)\cdot (2\pi)^{2k+2l} \cdot k!\cdot (l+1)!.
\end{equation}
Hence, we obtain that 
\begin{equation}
\begin{aligned}
 \delta_r(F,G)^2 &\lesssim \|u\|_{L^2(Q)} \cdot \left(\|F\|_{\F^\infty}^2 + \|G\|_{\F^\infty}^2 \right) \cdot \left(\sum_{k\in\N} \omega_k \cdot (2\pi)^{2k}(k+1)! \right)^2\\
 &\lesssim \|u\|_{L^2(Q)} \cdot \left(\|F\|_{\F^\infty}^2 + \|G\|_{\F^\infty}^2 \right) \cdot r^4 e^{8\pi^2 r^2},
 \end{aligned}
\end{equation}
where, in the last inequality we used that 
\begin{equation}
    \begin{aligned}
     \sum_{k\in\mathbb{N}} \omega_k \cdot (2\pi)^{2k}(k+1)! 
     &= \sum_{k\in\N} \frac{\pi r^{2k+2}}{k!(k+1)!} (2\pi)^{2k} (k+1)!\\
     &= \pi r^2 \sum_{k\in \N} \frac{(2\pi r)^{2k}}{k!}\\
     &= \pi r^2 e^{4\pi^2 r^2}.
    \end{aligned}
\end{equation}
\end{proof}

To arrive at the main statement of this section it merely remains to combine our preparatory results.
\begin{proof}[Proof of Theorem \ref{thm:basicstabfock}]
The similarity between $F$ and $G$ on $B_r$ can be bounded by 
\begin{equation}
    \begin{aligned}
     \min_{\tau\in\torus} \|G-\tau F\|_{L^2(B_r)}^2
                &\overset{\text{Prop. \ref{prop:reldeltadist}}}{\lesssim} \frac{\delta_r(F,G)^2}{\|F\|_{L^2(B_r)}^2}\\
                &\overset{\text{Prop. \ref{prop:bounddelta}}}{\lesssim}  \frac{ r^4 e^{8\pi^2 r^2} \left(\|F\|_{\F^\infty}^2 + \|G\|_{\F^\infty}^2\right) \||F|^2-|G|^2\|_{L^2(Q)}}{\|F\|_{L^2(B_r)}^2}.
    \end{aligned}
\end{equation}
 Taking roots yields the claim.
\end{proof}

%%%%%%%%%%%%%%%%%%%%%%%%%%%%%%%%%%%%%%%%%%%%%%%%%%%%%%%%%%%%%%%%%%%%%%%%%%%%%%%%%

\subsection{General case}\label{sec:stabilitygeneral}
The aim of the present section is to modify the techniques developed in \cite{cheng2020stable} so to make them applicable to our setting.\\ 

We begin with with recalling a few graph theoretical concepts.
Suppose $(\graphg,w,\sigma)$ is a weighted graph.
Then, the \emph{adjacency matrix} $A$ is defined by 
\begin{equation}
    A(u,v)= \sigma ((u,v)), \quad u,v\in\graphv.
\end{equation}
The \emph{degree} of a vertex $v\in\graphv$ is given by $\dg(v)=\sigma\Big( \bigcup_{u\in\graphv} \{(v,u)\} \Big) = \sum_{u\in\graphv} A(u,v).$
The \emph{graph Laplacian} is defined as $L:=D-A$, where $D$ is the diagonal matrix with $D(v,v)=\dg (v), \,v\in\graphv$. 
With these notions at hand we can introduce yet another concept measuring connectivity.
\begin{defn}
The \emph{algebraic connectivity} of a weighted graph $(\graphg,w,\sigma)$, where we assume that $w$ is summable, is defined by 
\begin{equation}\label{eq:defalgcon}
\lambda_\graphg:=\inf_{\mathbf{z}\neq 0,\, \langle \mathbf{z}, \mathbf{1}\rangle=0} \frac{\mathbf{z}^* L \mathbf{z}}{\|\mathbf{z}\|_{\ell^2(\graphv,w)}^2},
\end{equation}
where $L$ is the graph Laplacian, $\mathbf{1}$ is the vector consisting of ones exclusively, and the inner product $\langle \cdot,\cdot\rangle$ is in $\ell^2(\graphv,w)$.
\end{defn}
Cheeger's inequality connects algebraic connectivity and Cheeger constant of a graph.
For a proof we refer to \cite[Theorem A.1]{cheng2020stable}.
\begin{thm}[Cheeger inequality]
Let $(\graphg,w,\sigma)$ be a weighted graph. 
Then it holds that 
\begin{equation}
    2 h_\graphg \ge \lambda_\graphg \ge \frac{h_\graphg^2}{2\delta_0},
\end{equation}
where $\delta_0:= \sup_{v\in\graphv} \frac{\dg(v)}{w(v)}$.
\end{thm}
Next, we set up the graph corresponding to recovering the phase information of a concrete signal $f$ at hand.
\begin{defn}
Let $f\in L^2(\R)$ and suppose $\mathcal{Q}$ is a finite collection of squares of side length $1$.
 With $(f,\mathcal{Q})$ we associate a finite graph denoted by 
 $\graphg=\graphg_{f,\mathcal{Q}}=(\graphv,\graphe)$, where the vertex set $\graphv$ is identified with $\mathcal{Q}$, i.e.
 $\graphv\simeq \mathcal{Q}$, and which is fully connected, i.e. $\graphe=\graphv\times\graphv.$
Moreover, to obtain a weighted graph we assign weights to both the vertices and the edges by defining \footnote{here we implicitly identify $v\sim Q$ and $v'\sim Q'$}
\begin{equation}
 w(v):= \|\G f\|_{L^2(Q)}^2 = \|\Sp f\|_{L^1(Q)}, \quad v \in \graphv,
\end{equation}
as well as
\begin{equation}\label{def:sigmagraph}
 \sigma(e):= \|\G f\|_{L^2(Q\cap Q')}^4= \|\Sp f\|_{L^1(Q\cap Q')}^2, \quad e=(v,v')\in \graphe.
\end{equation} 
\end{defn}

Now that we have settled the terminology and concepts required, we are able to formulate the main result of the present section.
\begin{thm}\label{thm:generalcase}
 Let $f\in L^2(\R)$ and let $\mathcal{Q}$ be a finite family of squares of side length $1$, and set $\Omega:=\bigcup_{Q \in\mathcal{Q}} Q$.
 There exists a universal constant $C>0$ (independent from $f,\mathcal{Q}$) such that for all $g\in L^2(\R)$ it holds that
 \begin{multline}
  \min_{\tau\in\torus} \|\G g- \tau \G f\|_{L^2(\Omega)} \\ \le C 
  \left(KM^{1/2}L^{1/2} + \lambda^{-1}K\nu^{3/2}L^{1/2}+ \vol(\Omega)^{1/2}\right)^{1/2} 
  \|\Sp f - \Sp  g\|_{L^2(\Omega)}^{1/2},
 \end{multline}
 where $\lambda$ denotes the algebraic connectivity of the graph associated to $(f,\mathcal{Q})$, $\nu=|\mathcal{Q}|$ and where
 \begin{gather*}
      M=M(\Sp f,\mathcal{Q}):= \sum_{Q\in\mathcal{Q}} \|\Sp f\|_{L^1(Q)}^{-2},\\
      L=L(\mathcal{Q}):= \|\sum_{Q\in \mathcal{Q}} \mathbbm{1}_Q \|_{L^\infty(\R^2)},\\
      K=K(\Sp f, \Sp g):= \|\Sp f\|_{L^\infty(\R^2)} + \|\Sp g\|_{L^\infty(\R^2)}.
 \end{gather*}
\end{thm}
\begin{rem}
The main reason why we cannot directly apply the results established by Cheng \etal \cite{cheng2020stable}, is the absence of uniform stability of the local phase reconstruction process in our setting.
However, Theorem \ref{thm:specialcasebasicgabor} allows us to control 
the local stability constant by $\|\Sp f\|_{L^1(Q)}^{-1/2}=\|\G f\|_{L^2(Q)}^{-1}$, essentially.
We contain the lack of local uniform stability by defining the edge weights slightly different as compared to \cite{cheng2020stable}: 
By using \eqref{def:sigmagraph}, instead of $\sigma(e)=\|\G f\|_{L^2(u\cap v)}^2$, we compensate for the dependence of the local stability on the energy of the signal on the respective square.
\end{rem}
By making use of Cheeger inequality we may rewrite Theorem \ref{thm:generalcase} in terms of the Cheeger constant of the associated graph.
\begin{thm}\label{thm:generalcasecheeger}
Let $f\in L^2(\R)$ and let $\mathcal{Q}$ be a finite collection of unit squares.
Let $\Omega, \nu, M, L, K$ be defined as in Theorem \ref{thm:generalcase}.
Let $h=h_\graphg$ be the Cheeger constant of graph associated with $(f,\mathcal{Q})$ and $\delta_0:=\max_{v\in\graphv} \dg(v)/w(v).$\\
There exists a universal constant $C>0$ such that for all $g\in L^2(\R)$ it holds that
\begin{multline}
    \min_{\tau\in\torus} \|\G g- \tau \G f\|_{L^2(\Omega)} \\\le C  \left(KM^{1/2}L^{1/2} + h^{-2} \delta_0 K\nu^{3/2}L^{1/2}+ \vol(\Omega)^{1/2}\right)^{1/2}
  \|\Sp f - \Sp  g\|_{L^2(\Omega)}^{1/2}.
\end{multline}
\end{thm}
To establish Theorem \ref{thm:generalcase} we closely follow the proof methods in \cite{cheng2020stable}.
We collect a couple of auxiliary results.
The first one allows us to drop the constraint in the similarity measure.
\begin{lem}\label{lem:dropconstraint}
 For all $f,g\in L^2(\R)$ and measurable $D\subseteq \R^2$ we have that 
 \begin{equation}
  \min_{\tau\in\torus} \|\G g -\tau \G f\|_{L^2(D)} \le \sqrt{2} \min_{c\in\C} \| \G g- c\G f\|_{L^2(D)} + \||\G g|-|\G f|\|_{L^2(D)}. 
 \end{equation}
\end{lem}
\begin{proof}
 We recycle the proof of \cite[Lemma 4.2]{cheng2020stable}.
 One just needs to replace $\ell^2$ norms and inner product with the correct $L^2$ norms and inner products, respectively.
\end{proof}
Next, we have a closer look at the overlap of two components.
This is where it becomes apparent, why we have defined the edge weight $\sigma$ as in \eqref{def:sigmagraph}.
\begin{lem}\label{lem:overlap}
 Let $f,g\in L^2(\R)$ and let $Q_1,Q_2\subseteq\R^2$ be two squares, each of sidelength $1$.
 Moreover, for $j\in\{1,2\}$ let $c_j\in\mathbb{T}$ be a minimizer of $\|\G g- c_j \G f\|_{L^2(Q_j)}$.
 Then it holds that 
 \begin{equation}
  |c_1-c_2|^2 \|\G f\|_{L^2(Q_1\cap Q_2)}^4 \lesssim 
  \left(\|\Sp f\|_{L^\infty(\R^2)} + \|\Sp g\|_{L^\infty(\R^2)}\right) 
  \cdot \|\Sp f - \Sp g\|_{L^2(Q_1\cup Q_2)}.
 \end{equation}
\end{lem}
\begin{proof}
 First rewrite the term to be bounded
 \begin{equation}
   |c_1-c_2|^2 \cdot \||\G f\|_{L^2(Q_1\cap Q_2)}^4 
            = \|(c_1-c_2)\G f\|_{L^2(Q_1\cap Q_2)}^2 \cdot \|\G f\|_{L^2(Q_1\cap Q_2)}^2.
 \end{equation}
The square root of the first factor can be estimated by making use of the local stability result, Theorem \ref{thm:specialcasebasicgabor}.
\begin{equation}
\begin{aligned}
 \|(c_1-c_2)\G f\|_{L^2(Q_1\cap Q_2)} 
        &\le \|\G g- c_1\G f\|_{L^2(Q_1\cap Q_2)} + \|\G g- c_2\G f\|_{L^2(Q_1\cap Q_2)}\\
        &\le  \|\G g- c_1\G f\|_{L^2(Q_1)} + \|\G g- c_2\G f\|_{L^2(Q_2)}\\
        &\lesssim \left(\frac{\|\Sp f\|_{L^\infty(\R^2)} + \|\Sp g\|_{L^\infty(\R^2)}}{\min\left\{\|\Sp f\|_{L^1(Q_1)},\|\Sp f\|_{L^1(Q_2)}\right\}} \right)^{1/2} \\
        &\quad\times \left(\|\Sp f-\Sp g\|_{L^2(Q_1)}^{1/2} + \|\Sp f-\Sp g\|_{L^2(Q_2)}^{1/2} \right).
 \end{aligned}
\end{equation}
Next, observe that
\begin{equation}
 \|\G f\|_{L^2(Q_1\cap Q_2)}^2 = \|\Sp f\|_{L^1(Q_1\cap Q_2)} \le \min\{\|\Sp f\|_{L^1(Q_1)}, \|\Sp f\|_{L^1(Q_2)}\}.
\end{equation}
Moreover, since for $a,b\ge 0$ it holds that $(a^{1/2}+b^{1/2})^2 \asymp (a^2+b^2)^{1/2}$ we have that 
\begin{equation}
 \left(\|\Sp f-\Sp g\|_{L^2(Q_1)}^{1/2} + \|\Sp f-\Sp g\|_{L^2(Q_2)}^{1/2} \right)^2 \asymp \|\Sp f - \Sp g\|_{L^2(Q_1\cup Q_2)},
\end{equation}
which yields the desired estimate.
\end{proof}
We now turn towards the 
\begin{proof}[Proof of Theorem \ref{thm:generalcase}]
Let $\mathbf{z}=(z_v)_{v\in\graphv}$ be such that $z_v$ is a minimizer of $\zeta\in\C\mapsto\|\G g- \zeta \G f\|_{L^2(v)}$.\\

Let $c_0\in C$ be arbitrary but fixed. 
It follows from Lemma \ref{lem:dropconstraint} that 
\begin{equation}
\begin{aligned}
 \min_{\tau\in\torus} \|\G g- \tau \G f\|_{L^2(\Omega)}^2 
        &\lesssim \min_{c\in\C} \|\G g- c \G f\|_{L^2(\Omega)}^2 + \||\G g|-|\G f|\|_{L^2(\Omega)}^2\\
        &\le \|\G g- c_0\G f\|_{L^2(\Omega)}^2 +  \||\G g|-|\G f|\|_{L^2(\Omega)}^2.
\end{aligned}
\end{equation}
We proceed with estimating the first term on the right hand side and get that
\begin{equation}
 \begin{aligned}
  \|\G g- c_0\G f\|_{L^2(\Omega)}^2 
        &\le \sum_{v\in\graphv} \|\G g - c_0 \G f\|_{L^2(v)}^2 \\
        &\le \sum_{v\in\graphv} \|\G g -z_v \G f + z_v \G f - c_0\G f\|_{L^2(v)}^2\\
        &\lesssim \sum_{v\in\graphv} \left( \|\G g -z_v\G f\|_{L^2(v)}^2 + |z_v -c_0|^2 w(v)\right)
 \end{aligned}
\end{equation}
Thus we have for arbitrary $c_0\in \C$ that
\begin{multline}\label{est:3termstobebdd}
 \min_{\tau\in\torus}\|\G g- \tau \G f\|_{L^2(\Omega)}^2\\
 \lesssim 
 \sum_{v\in\graphv} \|\G g-z_v \G f\|_{L^2(v)}^2 
 + \sum_{v\in\graphv}|z_v-c_0|^2 w(v) 
 + \||\G g|-|\G f|\|_{L^2(\Omega)}^2
\end{multline}
We proceed by estimating each of the terms on the right hand side of \eqref{est:3termstobebdd}, one after another.\\

\textbf{First term:} We apply Theorem \ref{thm:specialcasebasicgabor} on each square and use Cauchy-Schwarz, to get
\begin{equation}
 \begin{aligned}
  \sum_{v\in\graphv} \|\G g- z_v\G f\|_{L^2(v)}^2 
  &\lesssim K \sum_{v\in\graphv} \|\Sp f\|_{L^1(v)}^{-1} \cdot \|\Sp f - \Sp g\|_{L^2(v)}\\
  &\le K \left(\sum_{v\in\graphv}\|\Sp f\|_{L^1(v)}^{-2} \right)^{1/2} \cdot \left( \sum_{v\in\graphv} \|\Sp f - \Sp g\|_{L^2(v)}^2 \right)^{1/2}\\
  &\le K M^{1/2} L^{1/2} \cdot \|\Sp f - \Sp g\|_{L^2(\Omega)}.
 \end{aligned}
\end{equation}
\textbf{Second term:}
First note that it follows from the Definition of the algebraic connectivity \eqref{eq:defalgcon} that 
\begin{equation}
 \sum_{v\in\graphv} |\mathbf{\zeta}_v|^2 w(v) \le \lambda^{-1} \cdot \mathbf{\zeta}^* L \mathbf{\zeta}, \quad \forall \mathbf{\zeta}\in \ell^2(\graphv,w)\setminus\{0\}: \langle\mathbf{\zeta},\mathbf{1}\rangle=0.
\end{equation}
In particular by picking $c_0:=\langle \mathbf{z},\mathbf{1}\rangle / \langle \mathbf{1},\mathbf{1}\rangle$, we can apply the above inequality on $\mathbf{\zeta}:=\mathbf{z}-c_0\mathbf{1}$. We therefore get that 
\begin{equation}
 \sum_{v\in\graphv}|z_v-c_0|^2 w(v) = \sum_{v\in\graphv} |\zeta_v|^2 w(v) \le \lambda^{-1} \cdot \mathbf{\zeta}^*L\mathbf{\zeta}.
\end{equation}
Since the adjacency matrix is symmetric
and with the help of Lemma \ref{lem:overlap}
we further get that 
\begin{equation}
\begin{aligned}
 \mathbf{\zeta}^* L \mathbf{\zeta} 
        &= \frac12 \sum_{u,v\in \graphv} |\zeta_u - \zeta_v|^2 \sigma((u,v))\\
        &= \frac12 \sum_{u,v\in \graphv} |z_u - z_v|^2 \sigma((u,v))\\
        &\lesssim K \sum_{u,v\in\graphv} \|\Sp f - \Sp g\|_{L^2(u\cup v)}\\
        &\le K \left(\sum_{u,v\in\graphv} 1 \right)^{1/2} \cdot \left( \sum_{u,v\in\graphv} \|\Sp f- \Sp g\|_{L^2(u\cup v)}^2\right)^{1/2}\\
        &= K \nu \left(  \int\limits_\Omega (\Sp f - \Sp g)^2 \left(\sum_{u,v\in\graphv} \mathbbm{1}_{u\cup v} \right) \,dA(z)\right)^{1/2}\\
        &\lesssim K \nu^{3/2} L^{1/2} \|\Sp f- \Sp g\|_{L^2(\Omega)} 
 \end{aligned}
\end{equation}
where we used that
\begin{equation}
 \big\|\sum_{u,v\in\graphv} \mathbbm{1}_{u\cup v} \big\|_{L^\infty(\R^2)}
 \le \big\|\sum_{u,v\in\graphv} \mathbbm{1}_u + \mathbbm{1}_v\big\|_{L^\infty(\R^2)} 
 \le 2\nu L.
\end{equation}

\textbf{Third term:}
The last term can be bounded as follows.
\begin{equation}
 \begin{aligned}
  \||\G f| - |\G g|\|_{L^2(\Omega)}^2 
        &= \int\limits_\Omega \left( |\G f|- |\G g|\right)^2 
        \cdot \frac{(|\G f|+|\G g|)^2}{(|\G f|+|\G g|)^2}\\
        &= \int\limits_\Omega \frac{(\Sp f - \Sp g)^2}{(|\G f|+ |\G g|)^2}\\
        &\le \|\Sp f - \Sp g\|_{L^2(\Omega)} \cdot \left( \int\limits_\Omega \frac{(\Sp f -\Sp g)^2}{(|\G f|+ |\G g|)^4}\right)^{1/2}\\
        &\lesssim \vol(\Omega)^{1/2} \|\Sp f- \Sp g\|_{L^2(\Omega)}.
 \end{aligned}
\end{equation}
Collecting the esimates gives the desired result
\begin{multline}
 \min_{\tau\in\torus} \|\G g- \tau \G f\|_{L^2(\Omega)}^2 \lesssim \left(KM^{1/2}L^{1/2} + \lambda^{-1} K \nu^{3/2}L^{1/2} + \vol(\Omega)^{1/2} \right)\\ \times  \|\Sp f- \Sp g\|_{L^2(\Omega)}.
\end{multline}

\end{proof}

%%%%%%%%%%%%%%%%%%%%%%%%%%%%%%%%%%%%%%%%%%%%%%%%%%%%%%%%%%%%%%%%%%%%%%%%%%%%%%%%%%%%%
\section{Discrete measurements}\label{sec:discretization}
We consider now the case where we only have access to the spectrogram on a discrete set of sampling points. 
More precisely, we aim to replace the 
term $\|\Sp f - \Sp g\|_{L^2(\Omega)}$ appearing on the right hand side of our fully continuous measurement setting from Section \ref{sec:fullycontinuous} 
by 
\begin{equation}
 \|\Sp f- \Sp g\|_{\ell^2_w(\Lambda)} = \left(\sum_{\lambda\in\Lambda} (\Sp f(\lambda) -\Sp g(\lambda))^2 w_\lambda \right)^{1/2}
\end{equation}
for a suitable set of sampling points $\Lambda$ with non-negative weights $(w_\lambda)_{\lambda\in\Lambda}.$
Since $\|\Sp f- \Sp g\|_{L^2}^2$ can be realized as the integral of a smooth function, utilization of cubature rules appears to be a promising strategy.
Following the pattern from Section \ref{sec:fullycontinuous} we will again work with squares as basic building blocks.
Throughout this section we will denote the square of side length $2s$ centered at the origin by $Q_s$, i.e. $Q_s=[-s,s]^2$.
Moreover, we will use $R_{a,b}=[-a,a]\times[-b,b]$ for rectangles.

\subsection{Product Gauss cubature rules for Analytic functions}
The integration schemes which we will employ arise from univariate Gauss rules.
\begin{defn}
For $s>0$ and $N\in \N$ let $(\gamma_k)_{k=1}^N$ denote the nodes of the $N$ point Gauss quadrature rule on $[-s,s]$, and let $(\omega_k)_{k=1}^N$ denote the corresponding weights.
The \emph{Product Gauss cubature rule} of degree $N$ on $Q_s$ is then defined as the scheme with node set 
\begin{equation}
    \Lambda=\Lambda^{(N,s)}:= \left\{ (\gamma_k,\gamma_l):\, 1\le k,l\le N\right\}
\end{equation}
with associated weights $w_\lambda=w_\lambda^{(N,s)}:=\omega_k \omega_l$.
\end{defn}

Given a continuous function $\phi$ on $Q_s$ we will define the integration error by 
\begin{equation}
    \mathcal{E}(\phi)= \mathcal{E}^{(N,s)}(\phi):= \iint\limits_{Q_s} \phi(x,y)\,dxdy - \sum_{\lambda\in \Lambda^{(N,s)}} \phi(\lambda) w_\lambda.
\end{equation}
If $Q$ is an arbitrary square of side length $s$ in the plane and $\phi$ a function defined on $Q$,
the product rule can be applied to $\phi$ by pulling the function back to $Q_s$ by virtue of an affine transformation $T$ with unit determinant.
In this case we will denote $\mathcal{E}_Q^{(N)}(\phi):= \mathcal{E}^{(N,s)}(\phi\circ T)$.\\

Typically the accuracy of a cubature formula depends on the smoothness on the function to be integrated.
We follow the approach by Chawla \cite{MR224280} and consider functions which have holomorphic extensions.

\begin{thm}\label{thm:chawlaadapted}
 Let $0<s<a$ and $b>0$. Suppose that $\Phi$ is holomorphic on a neighborhood of $R_{a,b}\times R_{a,b} \subseteq \C^2$ and set 
 \begin{equation}
     \phi(x,y):=\Phi(x+\mi 0 , y+\mi 0),\quad (x,y)\in Q_s.
 \end{equation}
 Then it holds for all $N\in \N$ that 
 \begin{multline}\label{est:chawlaadapted}
     \left| \mathcal{E}^{(N,s)} (\phi)\right|\\
     \le \frac{8s(a+b)}{\pi} \left(\frac{\min\{a-s,b\}}{s}\right)^{-N} \left(\frac{2(a+b)}{\min\{a-s,b\}}+\frac12 \log\left( \frac{a+s}{a-s}\right) \right)  \|\Phi\|_{L^\infty(E_{s,a,b})},
 \end{multline}
 where $E_{s,a,b}:= (R_{a,b}\times[-s,s]) \cup ([-s,s]\times R_{a,b})\subseteq \C^2.$
\end{thm}
\begin{proof}
First we prove the statement for the case $s=1$. The general case will then follow through a scaling argument.\\

As a starting point, we pick up formula (26') from \cite{MR224280} which states that
\begin{equation}\label{eq:chawla26prime}
    \mathcal{E}^{(N,1)}(\phi) = \frac2{\pi\mi} \int_\gamma \frac{Q_N(z)}{P_N(z)} \Phi(z,\zeta)\,dz + \frac{2}{\pi\mi} \int_{\gamma'} \frac{Q_N(u)}{P_N(u)} \Phi(\zeta',u)\,du,
\end{equation}
where $\gamma, \gamma'\subseteq\C$ are closed curves in the complex plane enclosing the interval $[-1,1]$, and where $\zeta,\zeta'\in [-1,1]$ are suitably chosen. 
Moreover, $P_N$ denotes the $N$-th Legendre polynomial, and $Q_N$ the $N$-th Legendre function of the second kind, single-valued and analytic in the complex plane with the interval $[-1,1]$ deleted.\\
We pick $\gamma=\gamma'=\partial R_{a,b}$, which clearly meets the above requirement.
It follows from \eqref{eq:chawla26prime} that 
\begin{equation}\label{eq:esterrorN1}
    \left|\mathcal{E}^{(N,1)}(\phi) \right| \le \frac{8(a+b)}\pi \|Q_N/P_N\|_{L^\infty(\gamma)} \|\Phi\|_{L^\infty(E_{1,a,b})}.
\end{equation}
In \cite[section 3.6.2., formula (22)]{MR0058756} we find the identity 
\begin{equation}
    Q_N(z) = P_N(z) \int\limits_z^{\infty} \frac{du}{(u^2-1)P_N(u)^2},
\end{equation}
where the path of integration does not cross the cut $[-1,1]$.\\
We consider now a fixed point $z_0\in \partial R_{a,b}$. Let $\beta$ denote the component of $\partial R_{a,b} \setminus( \{z_0\}\cup \{a+0\mi\})$ with minimal length.
With that, we get that 
\begin{equation}
    \begin{aligned}
     \left| Q_N(z_0)/P_N(z_0)\right| &= \left| \int_\beta \frac{dz}{(z^2-1)P_N(z)} + \int\limits_{a}^\infty \frac{dt}{(t^2-1)P_N(t)}\right|\\
                                    &\le \left(\inf_{z\in \partial R_{a,b} \cup (a,+\infty)} |P_N(z)| \right)^{-1}\\
                                    &\quad \times \left(\len(\beta) \left\|(\cdot^2-1)^{-1}\right\|_{L^\infty(\partial R_{a,b})} + \int\limits_a^\infty (t^2-1)^{-1}\,dt \right).
    \end{aligned}
\end{equation}
Since $\len(\beta)\le 2(a+b)$, and $\|(\cdot^2-1)^{-1}\|_{L^\infty(\partial R_{a,b})} \le \min\{a-1,b\}^{-2}$ and $\int\limits_a^\infty (t^2-1)^{-1}\,dt=\frac12 \log\left(\frac{a+1}{a-1}\right)$, we further have that
\begin{equation}
    \left| Q_N(z_0)/P_N(z_0)\right| \le \left(\inf_{z\in \partial R_{a,b} \cup (a,+\infty)} |P_N(z)| \right)^{-1} \cdot \left( \frac{2(a+b)}{\min\{a-1,b\}^2}+ \frac12 \log\left(\frac{a+1}{a-1}\right)\right).
\end{equation}
Since $z_0\in \partial R_{a,b}$ was arbitrary, we get with the help of Lemma \ref{lem:estlegendrepol} that
\begin{equation}
    \|Q_N/P_N\|_{L^\infty(\gamma)} \le \min\{a-1,b\}^{-N}  \cdot \left( \frac{2(a+b)}{\min\{a-1,b\}^2}+ \frac12 \log\left(\frac{a+1}{a-1}\right)\right).
\end{equation}
Now, combining this with \eqref{eq:esterrorN1} yields the claim for $s=1.$\\

Let now $s>0$ be arbitrary and suppose that $a>s$ and $b>0$.
We set $a':=a/s$ and $b':=b/s$. Moreover, let $\Phi':=\Phi(s\cdot)$ and $\phi':=\Phi'(\cdot+0\mi,\cdot+0\mi)$.
Using the first part we get that
\begin{equation}
    \begin{aligned}
     \left| \mathcal{E}^{(N,s)}(\phi)\right| &\overset{\text{c.o.v.}}{=} s^2 \left| \mathcal{E}^{(N,1)}(\phi')\right|\\
                                            &\overset{\text{part 1}}{\le} s^2 \frac{8(a'+b')}{\pi} \min\{a'-1,b'\}^{-N} \\ 
                                            &\qquad\times \left(\frac{2(a'+b')}{\min\{a'-1,b'\}}+\frac12 \log\left( \frac{a'+1}{a'-1}\right) \right) \|\Phi'\|_{L^\infty(E_{1,a',b'})}\\
                                            &= \frac{8s(a+b)}{\pi} \left(\frac{\min\{a-s,b\}}{s}\right)^{-N}\\
                                            &\qquad \times\left(\frac{2(a+b)}{\min\{a-s,b\}}+\frac12 \log\left( \frac{a+s}{a-s}\right) \right) \|\Phi\|_{L^\infty(E_{s,a,b})},
    \end{aligned}
\end{equation}
as desired.
\end{proof}

\subsection{Main results of this section}
Next we apply the error estimate, Theorem \ref{thm:chawlaadapted} to our integrand $(\Sp f- \Sp g)^2$.
\begin{prop}\label{prop:errorestspectrodiff}
 Let $Q\subseteq \R^2$ be a square of side length $s>0.$
 Then it holds for all $N\in\N$ and $f,g\in L^2(\R)$ that
 \begin{equation}
     \left| \mathcal{E}^{(N)}_Q ((\Sp f- \Sp g)^2) \right| \le 3 \left(\sqrt{8\pi}s+2 \right)^{N+3} N^{-\frac{N-1}2} e^{N/2}  (\|f\|_{L^2(\R)}^2 + \|g\|_{L^2(\R)}^2).
 \end{equation}
\end{prop}
Before we prove Proposition \ref{prop:errorestspectrodiff}, we discuss its implications.

\begin{thm}\label{thm:discretizationsquare}
 Assume the setting and conditions of Theorem \ref{thm:generalcasecheeger}. Moreover, assume that $\Omega$ is a square of side length $s>0$ and let $\varepsilon\in (0,1/2)$
 There exist universal constants $C,C'>0$ and a finite set $\Lambda\subseteq \Omega$ with associated weights $(w_\lambda)_{\lambda\in\Lambda} \subseteq \R_+$ satisfying $\sum_{\lambda\in\Lambda} w_\lambda=s^2$, such that 
 \begin{enumerate}[(i)]
 \item for all $g\in L^2(\R)$ it holds that 
 \begin{multline}
     \min_{\tau\in\torus} \|\G g-\tau\G f\|_{L^2(\Omega)}\\ \le C \left(KM^{1/2}L^{1/2} + h^{-2} \delta_0 K\nu^{3/2}L^{1/2}+ \vol(\Omega)^{1/2}\right)^{1/2} \left(\|\Sp f-\Sp g\|_{\ell^2_w(\Lambda)}^{1/2}+\varepsilon \right),
 \end{multline}
 \item the number of sampling points is bounded according to
 \begin{equation}\label{est:thmdiscretizationsquare}
     |\Lambda|\le C'  \cdot \max\left\{\ln \frac1{\varepsilon}, \ln \left(\|f\|_{L^2(\R)}^2+ \|g\|_{L^2(\R)}^2\right) , s\right\}^2.
 \end{equation}
 \end{enumerate}
\end{thm}
\begin{rem}
The assumption that the underlying domain $\Omega$ is a square is nothing more than convenient here. 
In principle one could also derive a similar result if $\Omega$ can be written as a union of squares, by applying Proposition \ref{prop:errorestspectrodiff} on each of these squares.
\end{rem}
\begin{proof}[Proof of Theorem \ref{thm:discretizationsquare}]
Once more, we set $\phi:=(\Sp f- \Sp g)^2$.
Unsurprisingly, the sampling points and weights employed will be coming from the Gauss product rule, with its degree $N$ yet to be specified.\\

Theorem \ref{thm:generalcasecheeger} tells us that 
\begin{multline}
     \min_{\tau\in\torus} \|\G g-\tau\G f\|_{L^2(\Omega)} \\ \le C \left(KM^{1/2}L^{1/2} + h^{-2} \delta_0 K\nu^{3/2}L^{1/2}+ \vol(\Omega)^{1/2}\right)^{1/2}  \|\phi\|_{L^2(\Omega)}^{1/2}
\end{multline}
Since we have that 
\begin{equation}
    \begin{aligned}
     \|\phi\|_{L^2(\Omega)}^{1/2} &= \left(\|\phi\|_{\ell^2_w(\Lambda)} + \mathcal{E}^{(N)}_\Omega (\phi) \right)^{1/4}\\
                                  & \le \left(\|\phi\|_{\ell^2_w(\Lambda)} + |\mathcal{E}^{(N)}_\Omega (\phi)| \right)^{1/4}\\
                                  & \le \|\phi\|_{\ell^2_w(\Lambda)}^{1/4} + |\mathcal{E}^{(N)}_\Omega(\phi)|^{1/4},
    \end{aligned}
\end{equation}  
it suffices to show that $|\mathcal{E}^{(N)}_\Omega (\phi)| \le \varepsilon^4.$
According to Proposition \ref{prop:errorestspectrodiff} (and after taking logarithms), this can be guaranteed by demanding that
\begin{equation}\label{eq:demandineq}
    \ln 3 + (N+3)\ln(\sqrt{8\pi}s+2) +\ln\kappa -\frac{(N-1)\ln N}2 +\frac{N}2 \overset{!}{\le} 4\ln \varepsilon,
\end{equation}
where we set $\kappa:=\|f\|_{L^2(\R)}^2+ \|g\|_{L^2(\R)}^2.$
Rearranging terms implies that \eqref{eq:demandineq} is equivalent to
\begin{equation}
    N \ln N- N -\ln N \overset{!}{\ge} 8\ln \frac1\varepsilon + 2\ln\kappa + 2(N+3)\ln(\sqrt{8\pi}s+2) + 2\ln 3.
\end{equation}
Since, $\ln N < N$, we have that $N\ln N - N - \ln N \ge (N-2)\ln N$. From this, one realizes that
\begin{equation}
    N\ge C' \max\left\{ \ln \frac1{\varepsilon}, \ln \kappa , s\right\},
\end{equation}
with $C'$ sufficiently large guarantees the desired accuracy. 
Since $|\Lambda|=N^2$ we get that indeed \eqref{est:thmdiscretizationsquare} holds true.
\end{proof}

\subsection{Proof of Proposition \ref{prop:errorestspectrodiff}}
The idea is to apply the error estimate provided in Theorem \ref{thm:chawlaadapted}. 
The proof consists of two crucial steps:
\begin{enumerate}[a)]
 \item Identify a function (in fact, such a function is unique if it exists) $\Phi$ which is analytic on $\C^2$ and which extends our integrand, i.e.,
 \begin{equation}\label{eq:Phiextprop}
  \Phi(x+0\mi,\xi+0\mi) = (\Sp f(x,\xi)-\Sp g(x,\xi))^2, \quad (x,\xi)\in Q.
 \end{equation}
 and estimate $\|\Phi\|_{L^\infty(E_{s,a,b})}$, with $E_{s,a,b}$ defined as in Theorem \ref{thm:chawlaadapted}.
\item Apply \eqref{est:chawlaadapted} for carefully chosen $a,b$.
\end{enumerate}

\textbf{Step a)}
We will find the following result to be rather useful.
\begin{lem}\label{lem:Tf}
 Let $f\in L^2(\R)$. Then 
 \begin{equation}
  \mathcal{T} f(z,\zeta):= \int_\R f(t) \exp \left\{-\pi(t-z)^2-2\pi\mi\zeta t \right\}\,dt,\quad (z,\zeta)\in\C^2
 \end{equation}
is an entire function and satisfies that 
\begin{equation}\label{eq:estTf}
\left| \mathcal{T} f(x+\mi y, \xi+\mi \eta)\right| \le 2^{-1/4} e^{\pi(y^2+2x\eta+\eta^2)}.
\end{equation}
\end{lem}
\begin{proof}
With $z=x+\mi y$ and $\zeta=\xi+\mi \eta$, 
Cauchy-Schwarz inequality gives that 
 \begin{equation}
 \begin{aligned}
  \int_\R \left| f(t) \exp \left\{-\pi(t-z)^2-2\pi\mi\zeta t\right\}\right|\,dt 
            &\le \|f\|_{L^2(\R)} \left(\int_\R \left| \exp \left\{-\pi(t-z)^2-2\pi\mi\zeta t\right\}\right|^2\,dt \right)^{1/2}\\
            &= \|f\|_{L^2(\R)} \left( \int_\R \exp\left\{-2\pi(t-x)^2+2\pi y^2 +4\pi\eta t \right\} \,dt\right)^{1/2}\\
            &= \|f\|_{L^2(\R)} \cdot 2^{-1/4} \exp\left\{\pi(\eta^2+2x\eta+y^2) \right\},
  \end{aligned}
 \end{equation}
 which implies \eqref{eq:estTf}.
 Moreover, with this estimate we can justify that differentiation and integration can be interchanged when differentiating $\mathcal{T} f$ w.r.t. $z$ and $\zeta$. Thus, $\mathcal{T} f$ is indeed entire.
\end{proof}
From its definition it is evident that $\mathcal{T} f$ extends $\G f$, meaning that $\T f(x+\mi 0,\xi+\mi 0)=\G f(x,\xi)$.
Thus, 
\begin{equation}
 \Phi(z,\zeta):= \left(\mathcal{T} f(z,\zeta) \overline{\mathcal{T} f(\bar{z},\bar{\zeta})} - \mathcal{T} g(z,\zeta) \overline{\mathcal{T} g(\bar{z},\bar{\zeta})}  \right)^2
\end{equation}
is a function with requested property \eqref{eq:Phiextprop}.
Furthermore, making use of \eqref{eq:estTf} yields
\begin{equation}
 \begin{aligned}
  |\Phi(z,\zeta)| &\le \left(\left| \mathcal{T} f(z,\zeta) \mathcal{T} f(\bar{z},\bar{\zeta})\right| + \left|\T g(z,\zeta) \T g(\bar{z},\bar{\zeta})\right| \right)^2\\
        &\le 2 \left( \left| \mathcal{T} f(z,\zeta) \mathcal{T} f(\bar{z},\bar{\zeta})\right|^2 + \left|\mathcal{T} g(z,\zeta) \mathcal{T} g(\bar{z},\bar{\zeta})\right|^2\right)\\
        &\le 2 \cdot 2^{-1/2} e^{4\pi(y^2+\eta^2)} \left(\|f\|_{L^2(\R)}^2+\|g\|_{L^2(\R)}^2\right)
 \end{aligned}
\end{equation}
Hence we obtain that 
\begin{equation}\label{est:PhiEsab}
 \|\Phi\|_{L^\infty(E_{s,a,b})} \le \sqrt{2} e^{4\pi b^2} (\|f\|_{L^2(\R)}^2 + \|g\|_{L^2(\R)^2}).
\end{equation}

\textbf{Step b)}
We can assume w.l.o.g. that $Q=Q_s.$
It follows from Theorem \ref{thm:chawlaadapted}, together with \eqref{est:PhiEsab} that 
\begin{multline}
    \left| \mathcal{E}^{(N,s)}((\Sp f - \Sp g)^2)\right|  \le 
     \frac{8s(a+b)}{\pi} \left(\frac{\min\{a-s,b\}}{s}\right)^{-N}\\
     \times \left(\frac{2(a+b)}{\min\{a-s,b\}}+\frac12 \log\left( \frac{a+s}{a-s}\right) \right) 
     \cdot \sqrt2 e^{4\pi b^2} (\|f\|_{L^2(\R)}^2 + \|g\|_{L^2(\R)}^2),
\end{multline}
where $a>s$ and $b>0$ are arbitrary.
Depending on $N$, we pick $b:=\sqrt{\frac{N}{8\pi}}$ and $a:=s+b$. With that, we get that  
\begin{multline}
         \left| \mathcal{E}^{(N,s)}((\Sp f - \Sp g)^2)\right|  \le \frac{\sqrt2 \cdot 8}{\pi} s(s+2b) \left(\frac{b}{s}\right)^{-N} \left( \frac{2(s+2b)}{b}+ \frac12\log\left(\frac{2s+b}{b}\right)\right)\\
         \times e^{4\pi b^2}(\|f\|_{L^2(\R)}^2 + \|g\|_{L^2(\R)}^2).
\end{multline}
Using that $\log(1+t)\le t$ for $t\ge 0$ allows us to further bound the above error by 
\begin{equation}
    \frac{\sqrt2 \cdot 8}{\pi} \left( \frac{2s(s+2b)^2}{b} + \frac{s^2(s+2b)}{b}\right) \left(\frac{b}{s}\right)^{-N} e^{4\pi b^2} (\|f\|_{L^2(\R)}^2 + \|g\|_{L^2(\R)}^2).
\end{equation}
Using that $b=\sqrt{\frac{N}{8\pi}}\ge (8\pi)^{-1/2}$  we get that 
\begin{equation}
   \left( \frac{2s(s+2b)^2}{b} + \frac{s^2(s+2b)}{b}\right) \le \frac{3s(s+2b)^2}{b} = 3bs \left(\frac{s}{b}+2 \right)^2 < 3b\left( \sqrt{8\pi}s+2\right)^3.
\end{equation}
Hence the error can be controlled as follows.
\begin{equation}
    \begin{aligned}
    \left| \mathcal{E}^{(N,s)}((\Sp f - \Sp g)^2)\right|  &\le \frac{\sqrt2 \cdot 24}{\pi \cdot \sqrt{8\pi}} \left( \sqrt{8\pi}s+2\right)^3 (\sqrt{8\pi}s)^N  N^{-\frac{N-1}2} e^{N/2}  (\|f\|_{L^2(\R)}^2 + \|g\|_{L^2(\R)}^2)\\
    & < 3 \left(\sqrt{8\pi}s+2 \right)^{N+3} N^{-\frac{N-1}2} e^{N/2}  (\|f\|_{L^2(\R)}^2 + \|g\|_{L^2(\R)}^2),
    \end{aligned}
\end{equation}
and we arrive at the desired result.

\section*{Acknowledgements}
MR gratefully acknowledges the support by the Austrian Science Fund (FWF) through the START-Project Y963-N35.

\printbibliography

\appendix
\section{Auxiliary results}
We estimate a couple of quantities involving the Gabor transform of a concrete pair of functions.
\begin{lem}\label{lem:estexsharp}
Let $\varphi=2^{-1/2} e^{-\pi\cdot^2}$. For $a>0$ we define 
\begin{equation}
    f=f_a:=\varphi(\cdot+a)+\varphi(\cdot-a),\quad\text{and}\quad g=g_a:=\varphi(\cdot+a)-\varphi(\cdot-a).
\end{equation}
Then it holds that 
\begin{enumerate}[i)]
\item with $z=x+\mi y$, the Gabor transforms of the two functions satisfy the identities
\begin{gather}
    \G f(x,y)= e^{-\frac\pi2 a^2}e^{-\frac\pi2 |z|^2-\pi\mi x y} \cos(a\pi\mi \bar{z}),\\
    \G g (x,y) = \mi e^{-\frac\pi2 a^2}e^{-\frac\pi2 |z|^2-\pi\mi x y} \sin(a\pi\mi \bar{z}).
\end{gather}
\item for all $a>0$ we have that 
\begin{equation}
        \frac{\min_{\tau\in\torus} \|\G g-\tau \G f\|_{L^2(Q)}}{\|\Sp f - \Sp g\|_{L^2(Q)}^{1/2} } \gtrsim e^{\pi a}.
\end{equation}
\end{enumerate}
\end{lem}
\begin{proof}
The Gabor transform of a shifted Gaussian can be elementary computed, i.e., 
\begin{equation}
    (\G \varphi(\cdot-\tau))(x,y) = e^{-\frac\pi2 [(x-\tau)^2+y^2]} e^{-\pi \mi (x+\tau)y}.
\end{equation}
Thus, since $f$ and $g$ only differ by the sign between the Gaussian components we get, that their respective Gabor transforms are equal to 
\begin{multline}
     e^{-\frac\pi2 [(x+a)^2+y^2]}e^{-\pi\mi(x-a)y} \pm e^{-\frac\pi2 [(x-a)^2+y^2]}e^{-\pi\mi(x+a)y}\\
     = e^{-\frac\pi2 a^2} e^{-\frac\pi2 (|z|^2-2\pi\mi xy)} \left( e^{-\pi a \bar{z}} \pm e^{\pi a \bar{z}} \right),
\end{multline}
which implies the first claim.\\

For the second assertion, we first estimate the denominator according to 
\begin{equation}
    \begin{aligned}
    \|\Sp f- \Sp g\|_{L^2(Q)}^2 &\asymp e^{-2\pi a^2} \left\||cos(a\pi\mi\bar{z})|^2 - |\sin(a\pi\mi\bar{z})|^2 \right\|_{L^2(Q)}^2\\
                                &\asymp e^{-2\pi a^2} \left\| \left|e^{a\pi\bar{z}}+e^{-a\pi\bar{z}} \right|^2 - \left|e^{a\pi\bar{z}}-e^{-a\pi\bar{z}} \right|^2 \right\|_{L^2(Q)}^2\\
                                & \asymp e^{-2\pi a^2} \left\| \Re \left\{e^{a\pi\bar{z}} \overline{e^{-a\pi\bar{z}}} \right\} \right\|_{L^2(Q)}^2\\
                                &= e^{-2\pi a^2} \left\| \Re \left\{e^{-2a\pi\mi y}\right\} \right\|_{L^2(Q)}^2\\
                                & \le e^{-2\pi a^2}.
    \end{aligned}
\end{equation} 
The nominator can be bounded from below by the estimate 
\begin{equation}
\begin{aligned}
    \min_{\tau\in\torus} \|\G g - \tau \G f\|_{L^2(Q)}^2 &\asymp e^{-\pi a^2} \min_{\tau\in\torus} \|\sin(a\pi\mi \bar{z}) - \tau \cos(a\pi\mi \bar{z})\|_{L^2(Q)}^2 \\
            &\gtrsim e^{-\pi a^2} \min_{c\in\C} \|\sin(a\pi\mi\bar{z}) - c\cos(a\pi\mi \bar{z}) \|_{L^2(B_{1/2})}^2.
\end{aligned}
\end{equation}
With $\phi:=\cos(a\pi\mi z)$ and $\psi:=\sin(a\pi\mi z)$ we have that 
\begin{equation}
\begin{aligned}
    \min_{c\in\C} \|\sin(a\pi\mi\bar{z}) - c \cos(a\pi\mi\bar{z})\|_{L^2(B_{1/2})}^2 &= \min_{c\in\C} \| \psi - c \phi \|_{L^2(B_{1/2})}^2 \\
    &= \|\psi\|_{L^2(B_{1/2})}^2  - \frac{|\langle \psi,\phi\rangle_{L^2(B_{1/2})}|^2}{\|\phi\|_{L^2(B_{1/2})}^2}
\end{aligned}
\end{equation}
Now from looking at the Taylor expansions of $\phi$ and $\psi$, and since complex monomials form a orthonormal system in $L^2(B_{1/2}) $ we find that $|\langle \psi,\phi\rangle_{L^2(B_{1/2})}=0$.
Moreover, again due to the monomials being pairwise orthogonal, we can estimate the norm of $\psi$ as follows.
\begin{equation}
    \begin{aligned}
     \|\psi\|_{L^2(B_{1/2})}^2 &\asymp \|e^{a\pi z}-e^{-a\pi z}\|_{L^2(B_{1/2})}^2 \\
                               &= 2\|e^{a\pi z}\|_{L^2(B_{1/2})}^2 - 2 \Re \left\{ \int_{B_{1/2} }e^{a\pi z - a \pi \bar{z}} \,dA(z)\right\}\\
                               &> 2\|e^{a\pi z}\|_{L^2(B_{1/2})}^2 - \frac\pi2
    \end{aligned}
\end{equation}  
Integration then yields that $\|e^{a\pi z}\|_{L^2(B_{1/2})}^2 = 0.405..\cdot e^{2\pi a}$. Thus we get that 
\begin{equation}
    \min_{\tau\in\torus} \|\G g - \tau \G f\|_{L^2(Q)}^2 \gtrsim e^{-\pi a^2+2\pi a}.
\end{equation}
Putting the two estimates together implies the claim.
\end{proof}
We require a bound for certain integrals.
\begin{lem}\label{lem:estgammalikeint}
 For all $p\in\N$ it holds that 
 \begin{equation}
  \int\limits_0^\infty r^{p+1} \exp\left\{-\frac{\pi}2 r^2+ \frac{\pi}{\sqrt{2}}r\right\}\,dr 
  < 2^{p+2}\cdot \Gamma\left(\frac{p}2+1\right).
 \end{equation}
\end{lem}
\begin{proof}
We rewrite the integral 
\begin{equation}
    \begin{aligned}
     \int\limits_0^\infty r^{p+1} \exp\left\{ -\frac{\pi}2 r^2 + \frac{\pi}{\sqrt{2}}r\right\}\,dr 
     &= \int\limits_0^\infty r^{p+r} \exp\left\{-\frac{\pi}2 \left(r-\frac1{\sqrt{2}} \right)^2 + \frac{\pi}4 \right\}\,dr\\
     &= e^{\pi/4} \int\limits_{-1/\sqrt{2}}^\infty \left(r+\frac{1}{\sqrt{2}}\right)^{p+1} e^{-\frac{\pi}2 r^2}\,dr\\
     &= e^{\pi/4} \sum\limits_{l=0}^{p+1} {p+1 \choose l} \left(\frac{1}{\sqrt{2}}\right)^{p+1-l} 
     \int\limits_{-1/\sqrt{2}}^\infty r^l e^{-\frac{\pi}2 r^2}\,dr
    \end{aligned}
\end{equation}

The integral on the right hand side can be bounded by
\begin{equation}
 \int\limits_{-1/\sqrt2}^\infty r^l e^{-\frac\pi2 r^2}\,dr \le 2\int\limits_0^\infty r^l e^{-\frac\pi2 r^2}\,dr = \frac{2 \Gamma \left(\frac{l+3}2 \right)}{(l+1) \cdot \left(\frac\pi2\right)^{\frac{l+1}2}} =  \frac{ \Gamma \left(\frac{l+1}2 \right)}{ \left(\frac\pi2\right)^{\frac{l+1}2}}
\end{equation}
With this we obtain that 
\begin{equation}
\begin{aligned}
     \int\limits_0^\infty r^{p+1} \exp\left\{ -\frac{\pi}2 r^2 + \frac{\pi}{\sqrt{2}}r\right\}\,dr  
     &\le e^{\pi/4} \sum\limits_{l=0}^{p+1} {p+1 \choose l} \left(\frac{1}{\sqrt{2}}\right)^{p+1-l} \frac{ \Gamma \left(\frac{l+1}2 \right)}{ \left(\frac\pi2\right)^{\frac{l+1}2}}\\
     &\le e^{\pi/4} \Gamma\left(\frac{p}2 +1\right) \cdot \left(\frac{1}{\sqrt{2}} + \sqrt{\frac{2}{\pi}}\right)^{p+1} \cdot \sqrt{\frac{2}\pi}\\
     &\le 4 \cdot 2^p \cdot \Gamma\left(\frac{p}2+1\right),
\end{aligned}
\end{equation}
which yields the desired bound.
\end{proof}

We need to control Legendre polynomials from below on the certain sets.
\begin{lem}\label{lem:estlegendrepol}
 Suppose that $a>1$ and $b>0$. Let $P_N$ denote the $N$-th degree Legendre polynomial.
 Then it holds that 
 \begin{equation}
     \inf_{z\in \partial R_{a,b} \cup (a,+\infty)} |P_N(z)|  \ge \min\{a-1,b\}^N.
 \end{equation}
 where $R_{a,b}=[-a,a]\times[-b,b]$.
\end{lem}
\begin{proof}
 First, we observe that $P_N$ is a polynomial of degree $N$, which has all its roots in the interval $(-1,1)$, the roots are symmetrically positioned and $P_N(1)=1.$
 Thus, we have that 
 \begin{equation}
     P_N(z)= z^\varepsilon \prod_{j=1}^{\lfloor \frac{N}2 \rfloor} \frac{z^2-t_j^2}{1-t_j^2},
 \end{equation}
 where $t_j$ are the positive roots of $P_N$, and where $\varepsilon\in\{0,1\}$ depending on whether $N$ is even or odd. 
 Moreover, since $P_N$ has all its roots in $(-1,1)$, we have that $x\mapsto |P_n(x)|$ is monotonically increasing on $[a,+\infty)$.
 Thus it suffices to show that 
 \begin{equation}\label{eq:lowerboundPNrectangle}
     \inf_{z\in \partial R_{a,b}} |P_N(z)|  \ge \min\{a-1,b\}^N.
 \end{equation}
 For $z\in\partial R_{a,b}$ and arbitrary $t\in(0,1)$ we can estimate that 
 \begin{equation}
     \left| \frac{z^2-t^2}{1-t^2}\right| = \frac{|z-t| \cdot |z+t|}{1-t^2} \ge \min\{a-1,b\}^2.
 \end{equation}
 Since $|z|\ge \min\{a-1,b\}$ for $z\in\partial R_{a,b}$ we get \eqref{eq:lowerboundPNrectangle}.
\end{proof}

\end{document}